\documentclass[10pt]{amsart}
\usepackage{amsfonts}
\usepackage{amssymb}
\usepackage{amsthm}
\usepackage{latexsym}
\usepackage{amsmath}
\usepackage{amssymb}

\usepackage[all,dvips]{xy}
\usepackage[dvips]{graphicx}

\newtheorem{teorema}{Theorem}[section]

\newtheorem{defin}[teorema]{Definition} %aquest seguir\`a la
\newtheorem{obs2}[teorema]{Remark}
\newtheorem{lema}[teorema]{Lemma}

\newtheorem{cor}[teorema]{Corollary}
\newtheorem{preg}[teorema]{Question}
\newtheorem{prop}[teorema]{Proposition}
\newenvironment{dem}{\begin{proof}[Proof]}{\end{proof}}
\newenvironment{skprf}{\begin{proof}[Proof (sketch)]}{\end{proof}}

\newcommand{\Z}{{\mathbb Z}}
\newcommand{\Q}{{\mathbb Q}}
\newcommand{\C}{{\mathbb C}}
\newcommand{\R}{{\mathbb R}}

\newcommand{\N}{{\mathbb N}}

\begin{document}

\title[Momose and quadratic points]{On quadratic points of classical modular curves}

\author{Francesc Bars}

\thanks{Partially supported by grant MTM2009-10359. }

%\date{23th April, 2012}

\maketitle

\begin{flushright}
    \textsc{Dedicated to the memory of F. Momose}
  \end{flushright}

\begin{abstract}
Classical modular curves are of deep interest in arithmetic
geometry. In this survey we show how the work of Fumiyuki Momose is
involved in order to list the classical modular curves which satisfy
that the set of quadratic points over $\Q$ is infinite. In
particular we recall results of Momose on hyperelliptic modular
curves and on automorphisms groups of modular curves. Moreover, we
fix some inaccuracies of the existing literature in few statements
concerning automorphism groups of modular curves and we make
available different results that are difficult to find a precise
reference, for example: arithmetical results on hyperelliptic and
bielliptic curves (like the arithmetical statement of the main
theorem of Harris and Silvermain in \cite{HaSi}, or the case $d=2$
of Abramovich and Harris theorem in \cite{abha}) and on the
conductor of elliptic curves over $\Q$ parametrized by $X(N)$.
\end{abstract}

\section{Introduction}

The propose of this survey is to present the relation of the work of
Professor Fumiyuki Momose with the determination of classical
modular curves which have an infinite set of quadratic points over
$\Q$.

Over the complex numbers, a classical modular curve $X_{\Gamma,\C}$
corresponds to a Riemann surface obtained by completing by the cusps
the affine curve $\mathbb{H}/\Gamma$ where $\mathbb{H}$ is the upper
half plane and $\Gamma$ is a modular subgroup of $SL_2(\Z)$.

For this survey let us consider the following modular subgroups
$\Gamma$ of $SL_2(\Z)$ with $N$ a positive integer and $\Delta$ a
strict subgroup of $(\Z/N\Z)^*$ with $-1\in\Delta$:
$$\Gamma(N)=\left\{\left(\begin{array}{cc} a&b\\
c&d\\
\end{array}\right)\in\mathrm{SL}_2(\mathbb{Z})\left|\left(\begin{array}{cc} a&b\\
c&d\\
\end{array}\right)\equiv \left(\begin{array}{cc} 1&0\\
0&1\\
\end{array}\right) \pmod N \right.\right\},$$
$$\Gamma_1(N)=\left\{\left(\begin{array}{cc} a&b\\
c&d\\
\end{array}\right)\in\mathrm{SL}_2(\mathbb{Z})\left|\left(\begin{array}{cc} a&b\\
c&d\\
\end{array}\right)\equiv \left(\begin{array}{cc} 1&*\\
0&1\\
\end{array}\right) \pmod N \right.\right\},$$
$$
\Gamma_{\Delta}(N)=\left\{\left(\begin{array}{cc} a&b\\
c&d\\
\end{array}\right)\in\Gamma_0(N)\left|\ (a\ \pmod N)\in \Delta \right.\right\},$$
$$\Gamma_0(N)=\left\{\left(\begin{array}{cc} a&b\\
c&d\\
\end{array}\right)\in\mathrm{SL}_2(\mathbb{Z})\left|\left(\begin{array}{cc} a&b\\
c&d\\
\end{array}\right)\equiv \left(\begin{array}{cc} *&*\\
0&*\\
\end{array}\right) \pmod N \right.\right\}.$$

Let us denote the associated Riemann surfaces (or complex modular
curves associated to the modular subgroups) by
$X(N)_{\C},X_1(N)_{\C},X_{\Delta}(N)_{\C}$ and $X_0(N)_{\C}$
respectively. Clearly $\Gamma(N)\leq \Gamma_1(N)\leq
\Gamma_{\Delta}(N)\leq \Gamma_0(N)$, therefore we have natural maps:
$$X(N)_{\C}\rightarrow X_1(N)_{\C}\rightarrow X_{\Delta}(N)_{\C}\rightarrow
X_0(N)_{\C}.$$

All these curves are algebraic and have as field of definition $\Q$.
They are curves associated to a modular problem with a $N$-level
structure, the modular problem over $\C$ can be used to give a model
of the curve over $\Q$ which is geometrically connected, this can be
done directly for $X_1(N)_{\C},X_{\Delta}(N)_{\C}$ and $X_0(N)_{\C}$
defining modular curves over $\Q$: $X_1(N),X_{\Delta}(N)$ and
$X_0(N)$, respectively, and each of them have at least one point
defined over $\Q$.

The usual complex modular problem associated to $X(N)_{\C}$ is the
coarse moduli spaces of the isomorphism classes of elliptic curves
$E$ together with two $N$-torsion points $P_1,P_2$ which satisfy:
the points $P_1,P_2$ generate $E[N]$ which is isomorphic to the
 group scheme $(\Z/N\Z)^2$ and $e(P_1,P_2)=\zeta_N$
where $e$ is the Weil pairing and $\zeta_N$ a fixed primitive
$N$-root of unity. In order to define a curve over $\Q$ with some
point over $\Q$ we need to modify the above modular problem by
$(E,\phi)$ with $\phi$ an isomorphism between $E[N]$ to the group
scheme $\Z/N\times\mu_N$ such that is equivariant by the Weil
pairing where $\mu_N$ is the group scheme of the $N$-roots of unity.
Denote by $X(N)$ the modular curve over $\Q$ by this modular
problem, which is isomorphic over $\C$ to the Riemann surface
$X(N)_{\C}$, (see \cite[\S2]{bkx} for more details concerning
$X(N)$).

Denote by ${X}_N$ any of the above geometrically connected modular
curves over $\Q$ with genus $\geq 2$.

By a great result of Gerd Faltings we know that for any number field
$F$ the set of $F$-points of ${X}_N$, named ${X}_N(F)$, is always a
finite set. Consider now the set of all quadratic points over $\Q$
of ${X}_N$:
$$\Gamma_2({X}_N,\Q):=\cup_{[F:\Q]\leq 2}{X}_N(F),$$
and ask if it is a finite set or not. With the work of Abramovich,
Silverman and Harris we know that $\Gamma_2({X}_N,\Q)$ is an
infinite set if and only if
 ${X}_N$ has a degree two map defined over $\Q$
to a projective line (hyperelliptic) or to an elliptic curve $E$
over $\Q$ (bielliptic) with positive rank, see \S 2 for the precise
general statement. In this survey we list all the $N$ where the
elements of the set $\Gamma_2({X}_N,\Q)$ is infinite.

The contents of the paper are as follows. In \S 2 we recall Harris
and Silverman results in \cite{HaSi} and Abramovich and Harris
results in \cite{abha}, in particular the relation between that the
set of quadratic points over some number field $L$ of a not singular
projective curve $C$ is infinite with the property that $C$ is a
hyperelliptic or a bielliptic curve. We also explain this relation
when one fixes the field $L$ fixing some inaccuracies in the
literature. In \S 3 we present and fix inaccuracies of some results
concerning the automorphism group of the above modular curves
${X}_N$ over the algebraic closure. In \S 4 we present results and
main ideas needed to obtain the exact list of $N$ where the number
of elements of the set $\Gamma_2(X_0(N),\Q)$ is infinite. Finally in
\S 5 we deal with the question on quadratic points over $\Q$ for the
remaining ${X}_N$, observing in particular that any elliptic curve
over $\Q$ parametrized by $X(N)$ satisfies that its conductor
divides $N^2$. We finish the survey with a remark on the modular
curves associated to the modular subgroup $\Gamma_1(M,N)$.

This paper corresponds to a written version of a lecture given at
the Boston-Barcelona-Tokyo seminar on Number Theory in honor of
Professor Fumiyuki Momose held in Barcelona from 23 to 25 May 2012.

\section{The set of quadratic points is infinite: Hyperelliptic and Bielliptic curves}
%\begin{frame}
Let $k$ be a number field, and $\overline{k}$ a fix algebraic
closure of $k$. We write $C=C_{|k}$ for a not singular projective
curve defined over $k$ and let $\overline{C}$ be
$C\times_k\overline{k}$. We denote by $g_C$ the genus of
$\overline{C}$ and we assume once and for all that $g_C\geq 2$. We
write $Aut(C)$ for the automorphism group of $C$ over
$\overline{k}$.

As usual $C(L)$ denotes the set of points of $C$ defined over $L$ or
$L$-points where $L$ is a finite field extension of $k$ inside
$\overline{k}$.

\begin{teorema}[Faltings, 1983] The set $C(L)$ is finite.
\end{teorema}

After Falting's result we can consider the quadratic points of $C$
over $L$ by:

$$\Gamma_2(C,L):=\cup_{[\ell:L]\leq 2}C(\ell),$$
where $\ell$ run over all the extensions of degree at most 2 of $L$
inside $\overline{k}$ and denote by $\#\Gamma_2(C,L)$ the number of
elements of this set.

And we can ask: When is $\#\Gamma_2(C,L)$ infinite?

\begin{defin} The curve $C$ is called hyperelliptic if
$\overline{C}$ admits a degree two morphism to the projective line
over $\overline{k}$.
\end{defin}
We have the following well-known result.
\begin{prop}\label{prop2.3} $C$ is hyperelliptic if and only if exists an involution $w\in
Aut({C})$ with $2g_C+2$ fixed points. This involution is unique if
exists and is defined over $k$ and we call it the hyperelliptic
involution.
\end{prop}

From the definition it follows easily
\begin{lema} If $C$ is hyperelliptic then exists a number field $L$
where $\#\Gamma_2(C,L)$ is infinite.
\end{lema}

Next, let us obtain an arithmetic result fixing the number field
$L$.

\begin{lema} \label{lem1}\label{Lem2.4} If $C$ is hyperelliptic, and
 $w$ denotes the hyperelliptic involution, then
\begin{enumerate}
\item there exists a (unique) degree two map to a conic, all
defined over $k$,
\item if $C(k)\neq \emptyset$, or, more generally, $C/<w>(k)\neq
\emptyset$, then  there exists a (unique) degree two map to
$\mathbb{P}^1_{|k}$, all defined over $k$.
\end{enumerate}
\end{lema}
\begin{dem} The curve $C/<w>$ has genus zero, therefore corresponds
to a conic, and because $w$ and $C$ is defined over $k$ the conic is
also defined over $k$.

Consider $\pi:C\rightarrow C/<w>$ a degree two morphism defined over
$\overline{k}$. For every $\delta\in Gal(\overline{k}/k)$ we have
the degree two morphism $\pi^{\delta}:C\rightarrow C/<w>$,. By the
uniqueness of the hyperelliptic involution the morphisms $\pi$ and
$\pi^{\delta}$ differs by an element $\xi_{\delta}\in
Aut(C/<w>)=PGL_2(\overline{k})$ because the conic is isomorphic to
the projective line in the algebraic closure. Thus we have an
application:

$$\xi: Gal(\overline{k}/k)\rightarrow PGL_2(\overline{k})$$
$$\delta\mapsto \xi_{\delta}.$$

Observe that given $\sigma,\delta\in Gal(\overline{k}/k)$ we have
$\xi_{\sigma\delta}=\xi_{\sigma}^{\delta}\circ\xi_{\delta}$ thus
$$\xi\in H^1(Gal(\overline{k}/k),PGL_2(\overline{k}))=0$$
therefore exists $\varphi_1\in Aut(C/<w>)$ satisfying
$\xi_{\sigma}=\varphi_1^{\sigma}\circ \varphi_1^{-1}$ for all
$\sigma\in Gal(\overline{k}/k)$. The morphism:

$$\varphi:=\varphi_1^{-1}\circ\pi:C\rightarrow C/<w>$$
is defined over $k$.

For the second statement, if $\varphi(C(k))\neq\emptyset$ or
$C/<w>(k)\neq\emptyset$ we obtain that the conic has a point in $k$,
therefore isomorphic to projective line over $k$
$\mathbb{P}^1_{|k}$.
\end{dem}
\begin{obs2}\label{Rk2.6} Mestre proved in \cite[p.322-324]{Me} that if $g_C$
is even and $C$ is defined over $k$ then exists
$\varphi:C\rightarrow \mathbb{P}^1_{|_k}$ all defined over $k$.
\end{obs2}
From Lemma \ref{Lem2.4} we obtain:
\begin{cor} If $C$ is hyperelliptic and $(C/<w>)(k)\neq\emptyset$ then
$\#\Gamma_2(C,k)$ is always infinite.
\end{cor}

 \begin{preg} \label{qu2.13} Let $C$ be
a not singular projective curve over $k$ with $(C/<w>)(k)=\emptyset$
and $C$ hyperelliptic (over $\overline{k}$), in particular $g_C$ is
odd by Remark \ref{Rk2.6}. Is it true that $\Gamma_2(C,k)$ is an
infinite set?

\end{preg}

The answer to this question is NO in general, see Corollary
\ref{cor2.13} and Lemma \ref{lem2.13}.

\begin{defin} A curve $C$ is called bielliptic if
$\overline{C}$ admits a degree two morphism to an elliptic curve
over $\overline{k}$.
\end{defin}
\begin{prop}\label{prop2.7} $C$ is bielliptic if and only if exists an involution $w\in
Aut({C})$ with $2g_C-2$ fixed points. The involution is unique if
$g_C\geq 6$ and then it is defined over $k$ and belongs to the
center of $Aut({C})$.
\end{prop}
See a proof of Proposition \ref{prop2.7} in \cite[p.706, Prop.1.2.a)
and Lemma 1.3]{Sch}.

\begin{cor} If $C$ is bielliptic then exists $L$ such that
$\#\Gamma_2(C,L)$ is not finite.
\end{cor}
\begin{dem}
Take $\varphi:C\rightarrow E$ a degree two morphism where $\varphi$
and $E$ are defined in some finite extension of $k$. Then take $L$
defined over some finite extension of $k$ such that $\varphi,E$ and
$rank E(L)\geq 1$ to conclude.
\end{dem}

Harris and Silverman obtain in \cite{HaSi}:

\begin{teorema}[Harris-Silverman]\label{thm2.10} Take $C$ with $g_C\geq 2$. Then:

$\exists L/k$ such that the set $\Gamma_2(C,L)$ is not finite
$\Leftrightarrow$ ${C}$ is a hyperelliptic or a bielliptic curve.
\end{teorema}

Let us state an arithmetic statement of Theorem \ref{thm2.10} fixing
the number field $L$.

We first recall the result \cite[Lemma 5]{HaSi} of J.Harris and
J.H.Silverman.

\begin{lema}\label{Lem2.11} Let $C$ be a bielliptic curve with $g_C\geq 6$. Then
exists a genus 1 curve $E$ defined over $k$ and a morphism
$\varphi:C\rightarrow E$ of degree 2 all defined over $k$. %Observe
%that if $C(k)\neq \emptyset$ then $E$ is immediately an elliptic
%curve over $k$.
\end{lema}

It is well-known to the specialists
%
%It can be deduced from the arguments of Abramovich and Harris in
%\cite{abha}
the following result (from the arguments of Abramovich and Harris in
\cite{abha}, or with our situation on quadratic points from Harris
and Silverman in \cite{HaSi}):
 \begin{teorema}\label{thm2.12} \mbox{}
Take $C$ with $g_C\geq 2$ then:

$\#\Gamma_2(C,k)=\infty$ if and only if ${C}$ is hyperelliptic with
a degree two morphism $\varphi:C\rightarrow \mathbb{P}^1_{|k}$
defined over $k$ to the projective line over $k$ or $C$ is
bielliptic with a degree two morphism $\phi:C\rightarrow E$ all
defined over $k$ where $E$ is an elliptic curve with $rank(E(k))\geq
1$.

%\end{enumerate}
%
%$$\#\Gamma_2(C,k)=\infty $$ $$\Updownarrow$$
%$ {C}$ is hyperelliptic with a degree two morphism
%$\varphi:C\rightarrow \mathbb{P}^1_{|k}$ defined over $k$ to the
%projective line over $k$ or $C$ is bielliptic with a degree two
%morphism $\phi:C\rightarrow E$ all defined over $k$ and the elliptic
%curve $E$ satisfies that  $rank(E(k))\geq 1$.
\end{teorema}
\begin{obs2} Schweizer, in \cite[proof of Theorem 5.1]{Sch} gives
different details of the proof of theorem \ref{thm2.12}. We warn
that the statement of \cite[Theorem 5.1]{Sch} is weaker than Theorem
\ref{thm2.12}.
\end{obs2}
\begin{skprf}
One implication of the statement of Theorem \ref{thm2.12} is clear.

Let us suppose that $\#\Gamma_2(C,k)=\infty$.

% The equivalence of the
%statements of (3) and (4)follows from Lemma \ref{lem2.14}, let us
%first introduce some notation.

Take $P\in \Gamma_2(C,k)$ a quadratic point and denote by $P'$ its
conjugate by the quadratic extension. Define then
$$\phi^{(2)}:S^2C\rightarrow Jac(C)$$
$$q_1+q_2\mapsto [q_1+q_2-P-P']$$
where $S^2C$ corresponds to $(C\times C)/S_2$ with $S_2$ the
permutation group and $Jac(C)$ is the Jacobian of $C$ and denote by
$proj$ the projection map $C\times C\rightarrow S^2C$. The map
$\phi^{(2)}$ is defined over $k$.

Denote by $\phi^{(2)}(k):S^2C(k)\rightarrow Jac(C)(k)$, the map on
the $k$-points, observe $\#S^2C(k)$ is infinite because
$\#\Gamma_2(C,k)=\infty$.

\begin{lema}\label{lem2.14} The map $\phi^{(2)}(k)$ is injective if and only
if does not exist a degree two map defined over $k$ of $C$ to the
projective line over $k$. In particular, if $C$ is hyperelliptic and
the genus zero curve $C/<w>$ has no $k$-point (where $w$ denotes the
hyperelliptic involution of $C$), then $\phi^{(2)}(k)$ is injective.
\end{lema}
\begin{dem}[of Lemma \ref{lem2.14}] Consider
$(q_1,q_2),(q_1',q_2')\in S^2C(k)$ and suppose
$$q_1+q_2-P-P'=q_1'+q_2'-P-P'\in Jac(C)(k)=Pic^0(C)(k).$$
We have that is equivalent to $q_1+q_2-q_1'-q_2'=div(f)$ with $f\in
k(C)$ of degree two. And this is equivalent to define a degree two
map over $k$ to the projective line over $k$.

In particular if $C$ with $g_C\geq 2$ is hyperelliptic, the
hyperelliptic involution is defined over $k$ and we have a degree
two morphism to $C$ to a genus 0 curve all defined over $k$, and if
the genus 0 curve has no points over $k$ (this situation only
happens with $g_C$ odd by Remark \ref{Rk2.6}) we have then that
$\phi^{(2)}(k)$ is injective.

\end{dem}

%Let us suppose that $\#\Gamma_2(C,k)=\infty$.
%%We may assume that $C$
%%is non-hyperelliptic, thus $g_C\geq 3$.
%Take $P\in \Gamma_2(C,k)$ a quadratic point and denote by $P'$ its
%conjugate by the quadratic extension. Define then
%$$\phi^{(2)}(k):S^2C(k)\rightarrow Jac(C)(k)$$
%$$q_1+q_2\mapsto [q_1+q_2-P-P']$$
%where $S^2C$ corresponds to $(C\times C)/S_2$ with $S_2$ the
%permutation group and $Jac(C)$ is the Jacobian of $C$ and denote by
%$proj$ the projection map $C\times C\rightarrow S^2C$.
%
%
%
%($C$ is non-hyperelliptic), and defined over $k$.

Now, we can suppose that $C$ is not hyperelliptic if $C(k)\neq
\emptyset$ or $C$ is hyperelliptic but satisfies that
$\phi^{(2)}(k)$ is injective. We will prove that under this
assumption $C$ is bielliptic with a degree two map $\varphi$ defined
over $k$ of the shape $\varphi:C\rightarrow E$ with $E$ an elliptic
curve with $rank(E(k))\geq 1$ proving theorem \ref{thm2.12}.

 By Falting's
Theorem \cite{Fa} we have
$$Im(\phi^{(2)}(k))=\cup P_i+B_i(k)$$
with $P_i$ points of $Jac(C)$ and $B_i$ abelian subvarieties of
dimension lower or equal to 1 because $Im(\phi^{(2)}(k))$ is not an
abelian variety. Therefore a $B_i$, say $B_1$ is an elliptic curve
$E$ where its $k$-points have positive rank, and $E$ is defined over
$k$ (see for more details \cite[proof Theorem 5.1]{Sch}). Now
Abramovich and Harris, in \cite[Lemma 2]{abha}, construct a degree
two map $\varphi$ from $C$ to $E$ as follows: take
$F:=\phi^{(2)}\circ proj:C\times C\rightarrow S^2C\rightarrow
Im(\phi^{(2)})$, $F$ is defined over $k$. $F^{-1}(E)$ is an union of
irreducible projective varieties with by $F$ are of degree 2 to $E$.
By \cite[Lemma 2]{abha} exists $Z_1$, an irreducible component of
$F^{-1}(E)$, and $j:Z_1\rightarrow C$ a rational map which is
one-to-one where $j$ corresponds to the projection of $C\times C$ to
$C$ in the first or the second component. In particular $j$ is
defined over $k$. Therefore $j$ induces an isomorphism $\iota$ over
$k$ from the normalization of $Z_1$, named $Z$, to $C$, and the
normalization map $\psi$ of $Z$ to $Z_1$ is defined over $k$.
Therefore the degree two map $\varphi:C\rightarrow E$ defined in the
proof of \cite[Lemma 2]{abha} is $\varphi:=F_{|_{Z_1}}\circ
\psi\circ\iota^{-1}$ and is defined over $k$.
\end{skprf}
\begin{cor} \label{cor2.13}  Any hyperelliptic curve $C$ defined
over $k$ with $(C/<w>)(k)=\emptyset$ and $g_C\geq 2$ which it is not
bielliptic over $k$ to an elliptic curve with positive rank
satisfies that $\Gamma_2(C,k)$ is a finite set
\end{cor}
\begin{proof} Because $C$ is hyperelliptic and $C/<w>(k)=\emptyset$ we
have that does not exist a degree two morphism of $C$ to the
projective line over $k$ by the uniqueness of the hyperelliptic
involution $w$. By the proof of theorem \ref{thm2.12} we conclude.
%$\Gamma_2(C,k)$ is not finite in this case if and only if $C$ is
%bielliptic over $k$ to an elliptic curve over $k$ with positive rank
%if and only if $S^2 C$ does not contains an elliptic curve over $k$
%with positive rank.
\end{proof}

\begin{lema}[Xarles]\label{lem2.13} Consider the curve $C$ in $\mathbb{P}^3$ given by the equations:
$y^2=-x^2-t^2$ and $z^2t^4=x^6+x^4t^2+x^2t^4+t^6$.

The curve $C$ is defined over $\Q$, has genus 5, satisfies
$C(\Q)=\emptyset$, is hyperelliptic with $C/<w>(\Q)=\emptyset$ and
it is not bielliptic over $\Q$ to an elliptic curve with positive
rank. In particular $\Gamma_2(C,\Q)$ is a finite set by Corollary
\ref{cor2.13},
 and the answer to question \ref{qu2.13} is No.
\end{lema}
\begin{proof} The quotient of $C$ by the automorphism $w:z\mapsto -z$ induces a
map of $C$ to the conic $y^2=-x^2-t^2$ with no points on $\Q$,
implying that $C$ is hyperelliptic and $C/<w>(\Q)=\emptyset$.

The Jacobian of $C$ is isogenous to the product of the Jacobian of
$C_1$ and $C_2$ where $C_1: t^2=x^6+x^4+x^2+1$ and $C_2:
t^2=(-x^2-1)(x^6+x^4+x^2+1)$. By use of Magma the Jacobian of $C_1$
and $C_2$ are $\Q$-simple of dimension 2 and 3 respectively,
justifying the genus and that there is no elliptic curve defined
over $\Q$ in $S^2 C$ and in particular in the Jacobian of $C$, thus
by the proof of theorem \ref{thm2.12} $C$ is not bielliptic to an
elliptic curve over $\Q$ with positive rank.
\end{proof}

 To finish this section we state the following result of Bob Accola and Alan Landman (see the
agreements in \cite{HaSi}) and also of Joe Harris and J.H. Silverman
in \cite[Prop.1]{HaSi}, very useful for the study of bielliptic
curves in a family of modular curves:

\begin{prop}[Accola-Landman, Harris-Silverman]\label{Pr2.13} If $C$ is a bielliptic curve and if $C\rightarrow C'$
is a finite map then the curve $C'$ is either bielliptic or
hyperelliptic.
\end{prop}

\section{Automorphism group of classical modular curves}

An important facet for modular curves $X_{N}$ with genus $\geq 2$ is
to compute the group $\mathrm{Aut}(X_{N})$ over the algebraic
closure. (We recall that if $C$ is bielliptic or hyperelliptic it
has a very special involution in $\operatorname{Aut\,}(C)$, we will
work with this fact in the next sections).

We recall that for a modular curve $X_{\Gamma,\C}$ with modular
group $\Gamma\leq \operatorname{SL}_2(\Z)$, the quotient of the
normalizer of $\Gamma$ in $\operatorname{PSL}_2(\mathbb{R})$ by
$\pm\Gamma$ gives a subgroup of
$\operatorname{Aut\,}(X_{\Gamma,\C})$, we denote this subgroup by
$Norm(\Gamma)/\pm\Gamma$.

This normalizer can be computed explicitly for different classical
modular curves.

For the modular curve $X_0(N)$, we have \cite{Ne}:
\begin{prop}[Newman]\label{prop3.1}  Write $N=\sigma^2 q$ with $\sigma,q\in\N$ and $q$ square-free.
Let $\epsilon$ be the $\operatorname{gcd}$ of all integers of the
form $a-d$ where $a,d$ are integers such that $\left(\begin{array}{cc} a&b\\
Nc&d\\
\end{array}\right)\in\Gamma_0(N)$. Denote by $v:=v(N):=\operatorname{gcd}(\sigma,\epsilon)$.
Then $M\in \operatorname{Norm}(\Gamma_0(N))/\pm\Gamma_0(N)$ if and
only if $M$ is represented in $PSL_2(\mathbb{R})$ as a matrix of the
form
$$\sqrt{\delta}\left(\begin{array}{cc} r\Delta&\frac{u}{v\delta\Delta}\\
\frac{sN}{v\delta\Delta}&l\Delta\\
\end{array}\right)$$
with $r,u,s,l\in\Z$ and $\delta|q$, $\Delta|\frac{\sigma}{v}$.
Moreover $v=2^{\mu}3^{w}$ with $\mu=min(3,[\frac{1}{2}v_2(N)])$ and
$w=min(1,[\frac{1}{2}v_3(N)])$ where $v_{p_i}(N)$ is the valuation
at the prime $p_i$ of the integer $N$.
\end{prop}
For later convenience we define particular elements in Proposition
\ref{prop3.1} which induce elements of $Aut(X_0(N))$.
\begin{defin}  Let $N$ be a fix positive integer. For every positive divisor $m'$ of $N$ with
$\operatorname{gcd}(m',N/m')=1$ the Atkin-Lehner involution $w_{m'}$
is defined as follows,
$$w_{m'}=\frac{1}{\sqrt{m'}}\left(\begin{array}{cc} m'a&b\\
Nc&m'd\\
\end{array}\right)\in SL_2(\R)$$
with $a,b,c,d\in\Z$.
\end{defin}
Always $w_d$ defines an involution of $Aut(X_0(N))$ and $w_d\cdot
w_{d'}=w_{dd'}\in Aut(X_0(N))$ with $(d,d')=1$.\\

Denote by $S_{v'}=\left(\begin{array}{cc} 1&\frac{1}{v'}\\
0&1\\
\end{array}\right)$ with $v'\in\N\setminus\{0\}$.

Atkin-Lehner claimed without proof in \cite[Theorem 8]{AL} the group
structure of $Norm(\Gamma_0(N))/\pm\Gamma_0(N))$, but their
statement is wrong. Later Akbas and Singerman \cite{AkSi}
(rediscovered also by the author in \cite{ba2}) obtain the correct
statement. Here we only present the following intermediate result:

\begin{prop}[Atkin-Lehner, Akbas-Singerman, Bars]\label{Bars} Any element $w$ that belongs to $\operatorname{Norm}(\Gamma_0(N))/\pm\Gamma_0(N)$
has an expression of the form
$$w=w_{m}\Omega,$$ where $w_m$ is an Atkin-Lehner involution of
$\Gamma_0(N)$ with $(m,6)=1$ and $\Omega$ belongs to the subgroup
generated by $S_{v(N)}$ and the Atkin Lehner involutions
$w_{2^{v_2(N)}}$, $w_{3^{v_3(N)}}$.
\end{prop}

We have also an analog of Newman's result for the modular curves
$X_1(N)$ in \cite{kiko}(or in \cite{La1}) which we state without
mentioning the particular case $X_1(4)$:
\begin{prop}[Kim-Koo, Lang]\label{normg1} If $N\neq 4$, then $M\in Norm(\Gamma_1(N))/\pm\Gamma_1(N)$ if and only
if $M$ is represented in $PSL_2(\mathbb{R})$ as a matrix of
determinant 1 of the form:
$$M=\frac{1}{\sqrt{Q}}\left(\begin{array}{cc}
Qx&y\\
Nz&Qw\\
\end{array}\right)$$
where $Q$ is a Hall divisor of $N$ (i.e. $(Q,N/Q)=1$) and $x,y,z,w$
are all integers.
\end{prop}

Concerning elements in $Aut(X_1(N))$ we observe the following
corollaries (see also \cite{jeki}):

\begin{cor} Always $w_Q\in Aut(X_1(N))$ for a Hall divisor $Q$ of $N$, but $w_Q$ is not necessarily an
involution of $X_1(N)$.  The full Atkin-Lehner involution $w_N$ is
always an involution of $Aut(X_1(N))$.
\end{cor}
\begin{cor} Consider $\gamma\in\Gamma_0(N)$ with
$\gamma\equiv\left(\begin{array}{cc} a&*\\
0&*\\ \end{array}\right)\ mod\ N$, then $\gamma$ represents an
automorphisms of $X_1(N)$ which depends only of $a$ and we name it
by $[a]$. In particular, $Norm(\Gamma_1(N))/\pm\Gamma_1(N)$ is
generated by $w_Q$ and $[a]$; i.e. generated by $\Gamma_0(N)/\pm1$
and $w_d$ with $d|N$ and $(d,N/d)=1$.
\end{cor}

For the classical modular curves $X(N)$ it follows easily following
the proof of $X_1(N)$ in \cite{kiko}, (or see also \cite{bkx}):

\begin{prop} If $N\geq 5$ the normalizer of $\Gamma(N)$ in
$PSL_2(\mathbb{R})$ is $PSL_2(\mathbb{Z})$ and therefore
$Norm(\Gamma(N))/\pm\Gamma(N)\cong PSL_2(\Z/N\Z)$.
\end{prop}

\begin{obs2}
 We mention
here that Mong-Lung Lang in \cite{La2} gave an algorithm that, given
a subgroup $\Gamma$ of finite index in $PSL_2(\Z)$ allows us to
determine the normalizer of $\Gamma$ in $PSL_2(\mathbb{R})$.
\end{obs2}

A very deep question is to determine when $Norm(\Gamma)/\pm\Gamma$
coincides with the full group of automorphisms of the corresponding
modular curve $X_{\Gamma,\C}$.

\begin{defin}
 An automorphism $v\in
\mathrm{Aut}(X_{\Gamma,\C})\setminus (Norm(\Gamma)/\pm\Gamma)$ is
called exceptional.
\end{defin}
It was known by A. Ogg \cite{Og}:
\begin{prop}[Ogg] If $p$ is a prime $p\neq 37$, then $X_0(p)$ has no
exceptional automorphisms. For $p=37$ $Aut(X_0(37))$ has index two
with $Norm(\Gamma_0(37))/\pm\Gamma_0(37)=\{id,w_{37}\}$ an one of
the exceptional automorphism corresponds to the hyperelliptic
involution of $X_0(37)$.
\end{prop}

F. Momose contributed strongly to this question, joint with K. Kenku
they proved the following very strong result \cite{KeMo}:

\begin{teorema}[Kenku-Momose]\label{thm3.1} Suppose that genus of $X_0(N)$ is $\geq 2$.
For $N\neq 37,63$ and $108$ there are no exceptional
automorphisms and therefore
$$Aut(X_0(N))=Norm(\Gamma_0(N))/\pm\Gamma_0(N).$$
\end{teorema}
\begin{obs2} The case $N=63$ was completed by Elkies in
\cite{Elkies} and the case $N=108$ was recently obtained by Harris
in \cite{Har}. In both cases there are exceptional automorphisms and
the index of the full group of automorphism to the normalizer is 2.
We warn to the reader that Kenku and Momose in \cite{KeMo} did not
discard the case $N=108$, see \cite{Har} for fix the gap of the
wrong argument did in \cite{KeMo}.
\end{obs2}

F. Momose also contributed to the question for the modular curves
$X_1(N)$ and $X_{\Delta}(N)$, under some conditions, unfortunately
\cite{Mo} is not available in the literature as far as I know. A
particular statement of the general result in \cite{Mo} reads as
follows:

\begin{teorema}[Momose]\label{Momax1} If $N$ is square-free then
$Aut(X_1(N))=Norm(\Gamma_1(N))/\pm\Gamma_1(N)$, i.e. no exceptional
automorphisms in $X_1(N)$ with $N$ square-free.
\end{teorema}
In \cite{Mo} it is stated the general result for $X_{\Delta}(N)$
with $N$ square-free (where $\Delta=(\Z/N)^*$ is allowed), but A.
Schweizer communicated us the following result which will appear in
the work \cite{jeki3}:

\begin{lema}[Jeon-Kim-Schweizer] \label{Lem3.9} The modular curve
$X_{\{\pm1,\pm6,\pm8,\pm10,\pm11,\pm14\}}(37)$, which we call
$X_{\Delta_3}(37)$, has exceptional automorphisms, and one of the
exceptional automorphisms is a bielliptic involution for this curve.
\end{lema}
\begin{dem}
$X_{\Delta_3}(37)$ has genus 4 and is an unramified Galois cover of
the genus 2 curve $X_0(37)$. So by \cite[Corollary 1, p. 346]{Ac2}
it is bielliptic.

The not exceptional automorphisms of $X_{\Delta_3}(37)$ form a group
$S_3$ generated by $[2]$ (of order 3) and the involution $w_{37}$.
The quotient by this group is the elliptic curve $X_0(37)/w_{37}$.
Applying the Hurwitz formula to this $S_3$-covering, one sees that
each of the 3 (conjugate) involutions has 2 fixed points. So the
bielliptic involution must be exceptional.

%This curve is of genus 4 and an unramified Galois cover of degree 3
%of the hyperelliptic curve $X_0(37)$.
%
%By \cite[Corollary 4.13,4.12]{Ac} , the (exceptional) hyperelliptic
%involution of $X_0(37)$ lifts to an involution of $X_{\Delta_3}
%(37)$ which can not be non-exceptional by the construction of the
%morphism $X_{\Delta_3}(N)\rightarrow X_0(N)$.
\end{dem}
\begin{obs2} A good revision of Momose work in \cite{Mo}
is needed in the literature to fix the square-free $N$ where
$X_{\Delta}(N)$ has exceptional automorphisms.
\end{obs2}
\begin{obs2}\label{Sch2}
 The proof that $X_{\Delta_3}(37)$ is not hyperelliptic in the
Tsukuba paper \cite{IsMo} relied on the nonexistence of exceptional
automorphisms which is not correct. To prove that $X_{\Delta_3}(37)$
is not hyperelliptic one can use that unramified Galois covers of
degree 3 of a hyperelliptic curve are never hyperelliptic
\cite[Lemma 4]{Ac2}.

One may wonder if a revision of \cite{IsMo} is needed concerning the
list of hyperelliptic modular curves, but the work of Jeon and Kim
in \cite{jeki2} \cite{jeki2b} already did this work. See in
particular \cite{jeki2b} for a more computational proof that
$X_{\Delta_3}(37)$ is not hyperelliptic.
\end{obs2}

Finally, we turn to the classical modular curves $X(N)$. It is known
to the specialists that there are no exceptional automorphisms, but
it was not an available proof in the literature until the recent
work of \cite{bkx}:

\begin{teorema} We have that $Aut(X(N))=Norm(\Gamma(N))/\pm \Gamma(N))$.
\end{teorema}

\section{Which curves $X_0(N)$ have an infinite set of quadratic points?}
%\begin{frame}
By Theorem \ref{thm2.10} it is enough to determine the modular
curves $X_0(N)$ which are hyperelliptic or bielliptic if we do not
wish to fix the number field $L$ where $\Gamma_2(X_0(N),L)$ is not a
finite set.

Andrew Ogg in \cite{Og} completed the list of hyperelliptic curves:

\begin{teorema}[Ogg]\label{thm4.1} There are 19 values of $N$,
such that $X_0(N)$ is hyperelliptic. For $N=37$, $N=40$ and $N=48$
the hyperelliptic involution is not of Atkin-Lehner type. The rest
can be read off from the following table:
\begin{center}
\begin{tabular}{|c|c||c|c|}
\hline
$N$&$\upsilon$&$N$&$\upsilon$\\
\hline
$22$&$w_{11}$&$35$&$w_{35}$\\
$23$&$w_{23}$&$39$&$w_{39}$\\
$26$&$w_{26}$&$41$&$w_{41}$\\
$28$&$w_{7}$&$46$&$w_{23}$\\
$29$&$w_{29}$&$47$&$w_{47}$\\
$30$&$w_{15}$&$50$&$w_{50}$\\
$31$&$w_{31}$&$59$&$w_{59}$\\
$33$&$w_{11}$&$71$&$w_{71}$\\
\hline
\end{tabular}
\end{center}
\end{teorema}

%\end{frame}
%\begin{frame}

In order to reduce to a finite list the $N$ for which $X_0(N)$ could
be a bielliptic curve: take $N$ with $g_{X_0(N)}\geq 6$ and $X_0(N)$
bielliptic, then by Lemma \ref{Lem2.11} we have a degree two map
$\varphi:X_0(N)\rightarrow E$ all defined over $\Q$ and if $p\nmid
N$, $p$ prime, we can reduce the map modulo $p$ obtaining a degree
two map of the reduction of $X_0(N)$ at $p$, denoted by
$X_0(N)_{|_{\overline{\mathbb{F}}_p}}$, to an elliptic curve over
the finite field $\mathbb{F}_p$.

Recall now that Ogg in \cite{Og} proved that any point $(E,C)\in
X_0(N)_{|_{\overline{\mathbb{F}}_p}}$ with $p\nmid N$ and $E$ a
supersingular curve over $\mathbb{F}_p$ satisfies that $(E,C)\in
X_0(N)_{|_{\overline{\mathbb{F}}_p}}(\mathbb{F}_{p^2})$, this fact
is the main item in order to prove in \cite{HaSi}:

\begin{teorema}[Harris-Silverman]\label{redHS} If $X_0(N)$ bielliptic (with genus $\geq
6$) then:
$$\#cusps\ in\
\mathbb{F}_{p^2}+2n(p)\mu(N)\leq\#X_0(N)_{|_{\overline{\mathbb{F}}_p}}(\mathbb{F}_{p^2})$$
$$\leq min(2(p+1)^2,p^2+pg_{X_0(N)}+1)$$

where $\mu(N)=(SL_2(\Z):\Gamma_0(N))$, and $n(p)=\sum_{E/F_p,\
supersingular}\frac{1}{|Aut(E)|}$.
\end{teorema}
We can control the cusps defined over $\mathbb{F}_{p^2}$ by
\cite[Prop.4.8]{Og2} (or \cite[Lemma 2.2]{ba}) and therefore we can
deduce for example if $2\nmid N$ then $N\leq 192$ (\cite[Lemma
2.1]{ba}), and,
\begin{prop}\cite[Prop.2.3]{ba} The curve $X_0(N)$ is not bielliptic for $N> 210$.
\end{prop}

%\end{frame}

%\begin{frame}

Now it remains a case-by-case study when $X_0(N)$ is bielliptic or
not with $N\leq 210$. By Proposition \ref{prop2.3} is enough to
control the fixed points of all involutions on $X_0(N)$.

 {\bf *} Because we know the number of fixed point on $X_0(N)$ of the Atkin-Lehner involutions $w_d\in
Norm(\Gamma_0(N))/\pm\Gamma_0(N)$ (see for example the formula in
\cite{Kl} or a table computing these numbers up to $N\leq 210$ in
\cite{STNB}):

\begin{lema} We can list all $N$ where $X_0(N)$ is bielliptic with an
involution of Atkin-Lehner type.
\end{lema}

{\bf *} Now by Theorem \ref{thm3.1} of Kenku and Momose and the
description of Newman in Proposition \ref{prop3.1} we obtain that
for $4\nmid N$ and $9\nmid N$ all the involutions are Atkin-Lehner
involutions.

%\end{frame}

%\begin{frame}

{\bf *} When $4|N$ or $9|N$ we haves more involutions in
$Aut(X_0(N))$ than the Atkin-Lehner type. For example when $4|N$
appears
the involution $S_2$.\\

In the paper \cite{ba} is done a detailed study of the involutions
which are not of Atkin-Lehner type that appear in the remaining
situations, in particular on the number of fixed points for these
new involutions. We reproduce next a particular statement
\cite[p.159-160]{ba}.

If $9||N$ and $4\nmid N$, from Proposition \ref{prop3.1}, one
obtains that
all involutions which are not of Atkin-Lehner type are:\\
$$S_3w_9S_3^2,\ \ S_3^2w_9 S_3$$
$$w_rS_3w_9S_3^2,\ \ w_rS_3^2w_9S_3\}\ \ if\ r\equiv 1(mod\ 3),$$
$$w_rS_3,\ \ w_rS_3^2,\ \ w_rw_9S_3w_9,\ \ w_rw_9S_3^2w_9\}\ \ if\ r\equiv 2(mod\ 3),$$
\begin{prop} The number of fixed points of $S_3w_9S_3^2$ and
$S_3^2w_9S_3$ in the modular curve $X_0(N)$ is equal to the number
of fixed points of $w_9$ in $X_0(N)$. \\
For $r\equiv 1(mod\ 3)$ the number of fixed points of
$w_rS_3w_9S_3^2, w_rS_3^2w_9S_3$ in $X_0(N)$ is equal to the number
of fixed points of $w_rw_9$ in $X_0(N)$.\\
For $r\equiv 2(mod\ 3)$ the number of fixed points of $w_rS_3,
w_rS_3^2, w_rw_9S_3w_9, w_rw_9S_3^2w_9$ is bounded by 3 times the
number of fixed points of $w_r$ in $X_0(N/3)$.
\end{prop}

After all this study on involutions we obtain in \cite[Theorem
3.15]{ba} all the bielliptic curves $X_0(N)$ with a list of
bielliptic involutions, in particular:

\begin{teorema}[Bars]\label{BieBars} There are 41 values of $N$, such that $X_0(N)$ is bielliptic.
For each value, $X_0(N)$ has a bielliptic involuton of Atkin-Lehner
type, except $X_0(2^3 3^2)$.
 The list of these $N$, $N\neq72$, is given below:
\begin{center}
\begin{tabular}{|c|c|c|c|c|c|c|c|c|c|}
\hline
22&26&28&30&33&34&35&37&38&39\\
\hline
40&42&43&44&45&48&50&51&53&54\\
\hline
55&56&60&61&62&63&64&65&69&75\\
\hline
79&81&83&89&92&94&95&101&119&131\\
\hline
\end{tabular}
\end{center}

\end{teorema}
\begin{obs2} Harrison noticed in 2011 that $X_0(108)$ has exceptional automorphisms. This result does not affect
the argument in \cite{ba} from 1999 because one discard $X_0(108)$
to become bielliptic by arguments using Theorem \ref{redHS}, see for
a precise proof \cite[Corol.lari 4.11]{batesin}. Harrison's result
\cite{Har} does not affect the conclusion of the works on bielliptic
curves for $X_1(N)$ or $X_{\Delta}(N)$ by Jeon and Kim because they
use proposition \ref{Pr2.13} with $C'=X_0(N)$ to discard $N=108$.
\end{obs2}

From Theorem \ref{thm4.1} and Theorem \ref{BieBars} we can list the
modular curves $X_0(N)$ with where the set of quadratic points over
some number field is not finite by Theorem \ref{thm2.10}:

\begin{cor} Assume $X_0(N)$ of genus greater than or equal to 2.
Then $$\#\Gamma_2(X_0(N),L_N)=\infty$$ for some number field $L_N$
if and only if $N$ is in the following list:\\
$22,23,26,28,30,31,33,34$,$35,37,38,39,40,41,42$,$43,44,45,46,47,48,50,$
$51,53$,$54$,\newline
$55$,$56,59,60$,$61,62,63,64,65,69,71$,$72,75,79,81,83,$
$89,92,94,95,101,119,131.$
\end{cor}

%\end{frame}

%\begin{frame}
If we want to determine $X_0(N)$ for which $\Gamma_2(X_0(N),\Q)$ is
not finite, we need to use Theorem \ref{thm2.12}. We obtained in
\cite{ba}
\begin{teorema}[Bars] Assume $g_{X_0(N)}\geq 2$.
We have that  $\Gamma_2(X_0(N),\Q)$ is finite if and only if $N$
does not appear in the following list
\begin{center}
\begin{tabular}{ccccc}
22&23&26&28&29\\
30&31&33&35&37\\
39&40&41&43&46\\
47&48&50&53&59\\
61&65&71&79&83\\
&89&101&131&\\
\end{tabular}
\end{center}
\end{teorema}
\begin{skprf}
Assume that $X_0(N)$ is bielliptic or hyperelliptic. We know that
$X_0(N)(\Q)\neq\emptyset$ (some cusps are defined over $\Q$), thus
we restrict to $X_0(N)$ bielliptic and no-hyperelliptic by Lemma
\ref{Lem2.4}. By Carayol's Theorem we discard $N$ where the elliptic
curves over $\Q$ with the conductor $N$ and all of its divisors have
rank zero. The remaining situations come from the study of the Weil
strong modular parametrization studied in \cite{MS}. (We mention
here that in the proof we use a weak form of Theorem \ref{thm2.12}
without worry if the degree two map from $X_0(N)$ to the elliptic
curve is defined over $\Q$. This is so because in $X_0(N)$
Atkin-Lehner involutions are defined over $\Q$ and for $4|N$ or
$9|N$ we obtain the result only searching when $Jac(X_0(N))$
contains an elliptic curve over $\Q$ with positive rank and if it
contains such elliptic curve then the theory of Weil strong modular
parametrization gives the morphism).
\end{skprf}
%\end{frame}

\section{Other classical modular curves}
%\begin{frame}
For simplify, we denote in the following $X_N'$, once and for all,
any of the modular curves:
$$X(N),\; \; X_1(N),\; or\; X_{\Delta}(N)\ (with\
\{\pm1\}\subset\Delta\leq(\Z/N)^*)$$ corresponding to modular groups
$\Gamma(N),\Gamma_1(N)=\Gamma_{\{\pm1\}}(N)$, or
$\Gamma_{\Delta}(N)$. And recall that always we assume $g_{X_n'}\geq
2$ and we have the natural finite maps:

$$X(N)_{\C}\rightarrow X_1(N)_{\C}\rightarrow X_{\Delta}(N)_{\C}\rightarrow
X_0(N)_{\C}.$$

\subsection{Hyperelliptic modular curves}\mbox{}\newline

%First we only state the results concerning which of the above
%modular curves are hyperelliptic curves, Momose contribution is
%crucial for these modular curves.

Consider the natural map $X_{N,\C}'\rightarrow X_0(N)_{\C}$. If
$X_N$ is hyperelliptic and $g_{X_0(N)}\geq 2$ then $X_0(N)$ is a
hyperelliptic curve, therefore by Theorem \ref{thm4.1} we are
reduced to study when $X_N'$ is hyperelliptic for a finite list of
$N$.

\begin{teorema}[Mestre, Ishii-Momose]\label{thm5.1} $X_1(N)$ is hyperelliptic only for
$N=13,16$ and $18$.
\end{teorema}
\begin{teorema}[Ishii-Momose, Jeon-Kim]\label{thm5.2} If $\{\pm1\}\varsubsetneqq\Delta\varsubsetneqq (\Z/N)^*$
the only hyperelliptic curve for the family $X_{\Delta}(N)$ is
$X_{\{\pm1,\pm8\}}(21)$.
\end{teorema}
\begin{obs2} Mestre listed the hyperelliptic modular curves $X_1(N)$ in \cite{Me2}.
Ishii and Momose studied again this problem for the modular curves
$X_{\Delta}(N)$ in \cite{IsMo}, reproving the result of Mestre
because $X_{\{\pm1\}}(N)=X_1(N)$. Ishii and Momose claimed in
\cite{IsMo} that always $X_{\Delta}(N)$ is not a hyperelliptic curve
when $\Delta\neq\{\pm 1\}$. Jeon and Kim proved that
$X_{\{\pm1,\pm8\}}(21)$ is hyperelliptic in \cite{jeki2b},
clarifying the gap in \cite{IsMo}.
\end{obs2}
\begin{obs2}
Let us make explicit that $w_d$ is not always in
$Norm(\Gamma_{\Delta}(N))$, and therefore is not an automorphism of
the curve $X_{\Delta}(N)$. This makes more difficult the work of
finding explicit involutions which are not exceptional for
$X_{\Delta}(N)$. This work began already in \cite{IsMo} and is
deeper developed in \cite{jeki3}.
\end{obs2}

By the work of Ishii and Momose in \cite{IsMo} (or with a more
direct proof by Bars, Kontogerogis and Xarles in \cite{bkx}):

%For the modular curves $X(N)$, by consider the natural map
%$\phi:X(N)\rightarrow X_1(N)$ and because $\phi$ maps an
%hyperelliptic curve maps to a hyperelliptic curve if
%$g_{X_1(N)}\geq 2$ joint with Theorem \ref{MIMom} and use of Hurwitz
%formula one observes \cite{bkx}:
\begin{teorema}\label{thm5.5} There are no $N$ for which $X(N)$ is hyperelliptic.
\end{teorema}

%\end{frame}

%\begin{frame} Second step: reduce to a finite set of $N$ where $X_{N}$

\subsection{Bielliptic modular curves}
\mbox{} \newline
%Now, we list the bielliptic curves in the above families and will
%give some indications how the results are attained.

By Proposition \ref{Pr2.13} with the natural map from $X_N'$ to
$X_0(N)$ and with Theorems \ref{thm4.1} and \ref{BieBars}:

\begin{cor} Take $N$ where $X_0(N)$ is not
bielliptic and is not hyperelliptic. Then $X_N'$ is not bielliptic,
in particular $X_N'$ is not a bielliptic curve for $N\geq 132$.
\end{cor}

%\end{frame}
For the families $X_1(N),X_{\Delta}(N),X(N)$ a detailed case-by-case
analysis of the involutions is developed in order to obtain the
exact list of $N$ for which $X_N'$ is a bielliptic curve, following
Proposition \ref{prop2.3}.
%\begin{frame}

\begin{teorema}[Jeon-Kim]\label{thm5.7} Take $N$ with $g_{X_1(N)}\geq 2$,
i.e. $N\geq 16$ or $N=13$. We have that the curve $X_1(N)$ is
bielliptic exactly when $2\leq g_{X_1(N)}\leq 6$ (this corresponds
to the values of $N$: $13,16,17,18,20,21,22,24$).
\end{teorema}
Let us present some main ideas given by Jeon and Kim in \cite{jeki}
to obtain the complete list of the curves $X_1(N)$ which are
bielliptic.

By proposition \ref{normg1} the not exceptional automorphisms of
$X_1(N)$ are of the form $[a]w_d$, and we need to study involutions
of this type. We warn again that $w_d$ is not necessarily an
involution of $X_1(N)$ but it is an automorphism.
\begin{itemize}
\item For $2\leq g_{X_1(N)}\leq 6$: Jeon and Kim prove that all $X_1(N)$ are
bielliptic with an involution in the normalizer of the modular
group. They observe that always $w_N$ is an involution of $X_1(N)$
and give some criteria for situations where the number of fixed
points of $w_N$ in $X_1(N)$ coincides with the number of fixed
points of $w_N$ in $X_0(N)$. Moreover they obtain some congruences
between $[a]w_d$ which allow the authors, (joint with the
hyperelliptic involutions obtained by Ishii and Momose in
\cite{IsMo} for $N=13,16,18$ and properties of automorphism groups
with hyperelliptic involution) to give a bielliptic involution
belonging to the normalizer of $\Gamma_1(N)$ for these $X_1(N)$.
\item For $g_{X_1(N)}>6$, the authors prove that all modular curves
$X_1(N)$ are not bielliptic. If $X_1(N)$ is bielliptic, by
proposition \ref{prop2.7} the bielliptic involution is defined over
$\Q$. The main argument is:\\
Let $v$ be the bielliptic involution of $X_1(N)$ and consider the
natural map to $X_0(N)$, thus we obtain an involution $\tilde{v}$ of
$X_0(N)$ and for $N\neq 37,63$ is not an exceptional automorphism of
$X_0(N)$, therefore the involution $v$ is not an exceptional
automorphism of $X_1(N)$. By Theorem \ref{normg1} $\tilde{v}$ is
necessarily of Atkin-Lehner type $w_d$. Restrict now from
\cite[Theorem 3.15]{ba} the couples $(N,w_d)$ where
$w_d=\frac{1}{\sqrt{d}}\left(\begin{array}{cc} dx&y\\
Nz&dw\end{array}\right)$ with $x,y,z,w\in\Z$ is a bielliptic
involution of $X_0(N)$. The argument of Jeon and Kim if $d\neq 2$ is
as follows:
 $\tilde{v}$ maps the cusp $\left(\begin{array}{c} 1\\
0\\
\end{array}\right)$ to $\left(\begin{array}{c} y\\
d\\
\end{array}\right)$ with $(y,d)=1$ for $d\neq N$ and to $\left(\begin{array}{c} 0\\
1\\
\end{array}\right)$ for $d=N$, but the cusps over $\left(\begin{array}{c} 1\\
0\\
\end{array}\right)$ are rational and the cusps over $\left(\begin{array}{c} y\\
d\\
\end{array}\right)$ or $\left(\begin{array}{c} 0\\
1\\
\end{array}\right)$ are not, therefore $v$ cannot exist.

To conclude the remaining cases, the authors use properties of
bielliptic involutions for example \cite[Prop.1.2(b)]{Sch}. We
remark that the argument in \cite{jeki} only uses Theorem
\ref{Momax1} for $N=37$.
\end{itemize}
%\end{frame}

%\begin{frame}
Now we turn to the the exact list of bielliptic modular curves
$X(N)$:
\begin{teorema}[Jeon-Kim, Bars-Kontogeorgis-Xarles] Take $N$
with $g_{X(N)}\geq 2$, i.e. $N\geq 7$. Then $X(N)$ is bielliptic for
and only for $X(7)$ or $X(8)$.
\end{teorema}
We sketch the main ideas of this Theorem following the approach
given in \cite{bkx}, see also remark
\ref{X1NMJK}. \\
Main points in getting the above result are:
\begin{itemize}
 \item $X(7)$ corresponds to the Klein quartic,
$X(8)_{\C}$ is the Wiman curve of genus 5 and both are known to be
bielliptic curves.\\
\item we have a morphism over $\Q$ $X(N)\rightarrow X_0(N^2)$ and from Theorem
\ref{BieBars} and Theorem \ref{redHS} $X(N)$ is not bielliptic for
$N\geq 10$. To deal with the case $N=9$ one uses a Hurwitz formula
argument for a convenient map with an explicit equation for
$X(9)_{\C}$.
\end{itemize}

%\begin{frame}

Recently Jeon, Kim and Schweizer completed the list of bielliptic
intermediate curves $X_{\Delta}(N)$ in \cite{jeki3}:

\begin{teorema}[Jeon-Kim-Schweizer, \cite{jeki3}]\label{thm5.9} Except for $N\neq 37$, the list of bielliptic
$X_{\Delta}(N)$ with $\{\pm1\}\subsetneqq\Delta\subsetneqq(\Z/N)^*$
are the following:\\
$X_{\{\pm1,\pm8\}}(21)$,$X_{\{\pm1,\pm5\}}(24)$,$X_{\{\pm1,\pm7\}}(24)$,
$X_{\{\pm1,\pm5\}}(26)$,\\$X_{\{\pm1,\pm3,\pm9\}}(26)$,
$X_{\{\pm1,\pm13\}}(28)$,$X_{\{\pm1,\pm3,\pm9\}}(28)$,\\$X_{\{\pm1,\pm4,\pm5,\pm6,\pm7,\pm9,\pm13\}}(29)$,
$X_{\{\pm1,\pm11\}}(30)$,
$X_{\{\pm1,\pm15\}}(32)$,\\$X_{\{\pm1,\pm2,\pm4,\pm8,\pm16\}}(33)$,
$X_{\{\pm1,\pm9,\pm13,\pm15\}}(34)$,$X_{\{\pm1,\pm6,\pm8,\pm13\}}(35)$,\\
$X_{\{\pm1,\pm4,\pm6,\pm9,\pm11,\pm16\}}(35)$,$X_{\{\pm1,\pm11,\pm13\}}(36)$,\\
%Falta la 37.(tot l'estudi...., en el preprint es tots els casos del 37 no sol el Delta_3)
%$X_{\{\mp1,\pm6,\pm8,\pm10,\pm11,\pm14\}}(37)$,???\\
$X_{\{\pm1,\pm4,\pm10,\pm14,\pm16,\pm17\}}(39)$, $X_{\{\pm1,\pm9,\pm
11,\pm19\}}(40)$,\\
$X_{\{\pm1,\pm2,\pm4,\pm5,\pm8,\pm9,\pm10,\pm16,\pm18,\pm20\}}(41)$,$X_{\{\pm1,\pm4,\pm11,\pm14,\pm16,\pm19\}}(45)$,\\
$X_{\{\pm1,\pm11,\pm13,\pm23\}}(48)$, $X_{\{\pm1,\pm6,\pm8,\pm13,\pm15,\pm20,\pm22\}}(49)$,\\
$X_{\{\pm1,\pm9,\pm 11,\pm19,\pm21\}}(50)$,$X_{\{\pm1,\pm4,\pm6,\pm9,\pm14,\pm16,\pm19,\pm21,\pm24,\pm26\}}(55)$,\\
$X_{\{\pm1,\pm7,\pm9,\pm15,\pm17,\pm23,\pm25,\pm31\}}(64)$.
\end{teorema}

Jeon, Kim and Schweizer, need new ideas to obtain the above Theorem.
We list some of the main new items in \cite{jeki3}:

{\bf *} Precise criteria to decide when $w_d$ gives an automorphism
in $X_{\Delta}(N)$ and when it defines an involution in
$X_{\Delta}(N)$ and in such case compute the number of fixed points.
Mainly, they prove that the bielliptic $X_{\Delta}(N)$ curves have a
bielliptic involution of type $w_d$ or $[a]w_d$.

{\bf *} To discard and obtain few new bielliptic curves they study
the natural morphism $X_{\Delta}(N)_{\C}\rightarrow X_0(N)_{\C}$ and
how a bielliptic involution on $X_{\Delta}(N)$ should translate to
$X_0(N)$ and compare with the known results on bielliptic
involutions on $X_0(N)$. In particular they need to check for
$X_{\Delta}(37)$ for different $\Delta$, if these curves have
exceptional automorphisms. They obtain, for example, that
$X_{\Delta_3}(37)$ has exceptional automorphisms (this contradicts a
result claimed in \cite{Mo} see also Lemma \ref{Lem3.9}) and  obtain
that this modular curve is bielliptic with an exceptional bielliptic
involution. The authors also need a particular study for genus 4
curves to obtain that $X_{\Delta}(25)$ is not bielliptic.
%\end{frame}

%\begin{frame}

\subsection{On infiniteness of quadratic points over $\Q$ of modular
curves}\mbox{}
\newline

Consider first the modular curves $X_1(N)$ that are defined over
$\Q$ with moduli interpretation given by couples $(E,C)$ with $C$ a
$N$-torsion subgroup of the elliptic curve $E$.

Fix $N$, then in order to have $X_1(N)$ a not finite set of
quadratic points over the rationals, we need an infinite set of
couples $(E,C)$ defined over a quadratic field, in particular $C$
should appear as torsion subgroup for elliptic curves over quadratic
fields.

We have the following main result of Kenku and Momose \cite{KeMo2}:
\begin{teorema}[Kenku-Momose] For $d$ a fix integer and $E$ an elliptic cuver over a quadratic field $\Q(\sqrt{d})$. Then
the torsion subgroup of $E(\Q(\sqrt{d}))$ is
isomorphic to one of the following groups:\\
$\Z/m\Z$ with $m\leq 16$ or $m=18$\\
$\Z/2\times\Z/2k$ with $k\leq 6$\\
$\Z/3\times\Z/3l$ with $l\leq 2$\\
$\Z/4\times\Z/4$.
\end{teorema}
Now by Theorem \ref{thm2.12} and because $X_1(13)$, $X_1(16)$ and
$X_1(18)$ are hyperelliptic (Theorem \ref{thm5.1}) we have after
Theorem \ref{thm5.7} the following result in \cite{jeki}:
\begin{cor} Take $N$ with $g_{X_1(N)}\geq 2$, i.e.
$N\geq 16$ or $N=13$. We have that  $\Gamma_2(X_1(N),\Q)$ is finite
if and only if $N$ does not appear in the following list:
$$13,16,18$$
\end{cor}

%\end{frame}

%\begin{frame}
Consider now the modular curves $X_{\Delta}(N)$ defined over $\Q$
and with natural morphisms
$$X_1(N)\rightarrow X_{\Delta}(N)\rightarrow X_0(N).$$

\begin{cor}\label{cor5.12} Take $N\neq 37$ with $g_{X_{\Delta}(N)}\geq 2$, and $\{\pm
1\}\subsetneqq \Delta\subsetneqq(\Z/N)^*$. Then the set
$\Gamma_2(X_{\Delta}(N),\Q)$ is not finite if and only if
$X_{\Delta}(N)=X_{\{\pm 1,\pm 8\}}(21)$.
\end{cor}
\begin{dem} By Theorem \ref{thm2.12} we only need to study
$X_{\Delta}(N)$ bielliptic and not hyperelliptic. Take the list of
$X_{\Delta}(N)$ bielliptic in Theorem \ref{thm5.9}, and suppose is
given $\phi: X_{\Delta}(N)\rightarrow E$ all defined over $\Q$. In
particular we have a morphism over $\Q$: $X_1(N)\rightarrow E$. By
Carayol's result we have that the conductor of $E$ should divide
$N$. Now for all the $N\neq 37$ for which $X_{\Delta}(N)$ is
bielliptic, we have $rank E(\Q)=0$ for all elliptic curves over $\Q$
with conductor dividing $N$ by Cremona tables \cite{Cre}.
\end{dem}

\begin{obs2} For $X_{\Delta}(37)$ we can not use the above Carayol
argument because we have the elliptic curve $37a$  $[ 0,0,1,-1,0]$
which has rank 1. Jeon, Kim and Schweizer in \cite{jeki3}, with an
ad-hoc proof, obtain that $\Gamma_2(X_{\Delta}(37),\Q)$ is always a
finite set for any $\Delta$ as above.
\end{obs2}

Now consider the modular curves $X(N)$.

Here we could try to reproduce the proof of Corollary \ref{cor5.12}
for ${X}(N)$ but we warn that Carayol result does not apply, for
example one can construct a map
$${X}(8)\rightarrow \{y^2=x(x-1)(x+1)\}=E$$
and $E$ has conductor $32$, not a divisor of $8$ and the rational
model of $X(8)$ is the one in \cite{ya} (see \cite{bkx} to ensure
that the model of $X(8)$ in \cite{ya} coincides with the one fixed
in this paper).

\begin{prop}\label{lemCar} Assume that we have a morphism $\varphi$ over $\Q$:
$$\varphi:{X}(N)\rightarrow E$$
with $E$ an elliptic curve defined also over $\Q$. Then the
conductor of $E$ divides $N^2$.
\end{prop}
\begin{dem} By \cite[Lemma1]{bkx} we have natural morphisms defined over
$\Q$: $X_1(N^2)\rightarrow X(N)$ and $X(N)\rightarrow X_0(N^2)$,
therefore we have a morphism over $\Q$ $X_1(N^2)\rightarrow E$ and
by Carayol we obtain that the conductor of $E$ divides $N^2$.
\end{dem}

%\begin{lem}\label{lemCar} Take $\mathcal{X}(N)_{\Q}$ with
%$\mathcal{X}(N)_{\Q}(\Q)\neq\emptyset$ and
%$g_{\mathcal{X}(N)_{\Q}}\geq 6$. Suppose that we have a morphism
%over $\Q$ of degree two $\mathcal{X}(N)_{\Q}\rightarrow E$ with $E$
%an elliptic curve over $\Q$. Then the conductor of $E$ divides
%$N^2$.
%\end{lem}
%\begin{dem} Observe first that $\mathcal{X}(N)_{\Q}$ is birational to
%$X_{\Delta}(N^2)$ by remark \ref{X1NMJK} with a particular $\Delta$.
%Thus we have a degree two map from $X_{\Delta}(N^2)$ to $E$ all
%defined over $\Q$ because of Lemma \ref{Lem2.11}. Thus we have a map
%$X_1(N^2)\rightarrow X_{\Delta}(N^2)\rightarrow E$ and by Carayol's
%result the conductor of $E$ divides $N^2$.
%\end{dem}

We can proof by an ad-hoc method the following result in \cite{bkx}
(without use of Proposition \ref{lemCar}):

\begin{teorema}[Bars-Kontogeorgis-Xarles] For any $N\geq 7$ we have that the set
$\Gamma_2(X(N)\times_{\Q}\Q(\zeta_N),\Q(\zeta_N))$ is always a
finite set, and in particular $\Gamma_2({X}(N),\Q)$ is always
finite.
\end{teorema}

\begin{obs2}[The family of modular curves $X_1(N,M)$]\label{X1NMJK}
For positive integers $M|N$ consider the congruence subgroup
$\Gamma_1(M,N)$ of $SL_2(\Z)$ defined by:
$$\Gamma_1(M,N):=\left\{A=\left(\begin{array}{cc}
a&b\\c&d\\
\end{array}\right)\in SL_2(\Z)|A\equiv \left(\begin{array}{cc}
1&*\\0&1\\
\end{array}\right)\ mod\ N,\ M|b\right\}.$$
Let $X_1(M,N)_{\C}$ be the Riemann surface associated to
$\Gamma_1(M,N)$. It is the coarse moduli space of the isomorphism
classes of elliptic curves $E$ with a pair $(P_M,P_N)$ of points
$P_M$ and $P_N$ of $E$ which generate a subgroup isomorphic to
$\Z/M\oplus\Z/N$ and satisfy $e_{E}(P_M,\frac{N}{M}P_N)=\zeta_M$
where $e_E$ is the Weil pairing and $\zeta_M$ a fixed $M$-th root of
unity, is known that the field of definition of the Riemann surface
$X_1(M,N)_{\C}$ is $\Q$ and the above coarse moduli problem gives a
model over $\Q(\zeta_M)$ and denote by $X(M,N)$ this model
corresponding to a curve defined over $\Q(\zeta_M)$.

Observe that $X_1(M,N)$ is birational, over some number field, to
$X_{\Delta}(MN)$ with
$\Delta=\{\pm1,\pm(N+1),\pm(2N+1),\ldots,\pm((M-1)N+1)\}$.

Jeon and Kim in \cite{jeki2} listed all $X_1(M,N)$ which are
bielliptic (obtaining a lot of the results by working in
$X_{\Delta}(MN)$ with arguments already exposed in the survey), and
they determine in \cite[Theorem 4.5]{jeki2} a list of $X_1(M,N)$
with an infinite set of quartic points over $\Q$, i.e. points of
degree 4 over $\Q$ (because the model of $X(M,N)$ is over
$\Q(\zeta_M)$ in \cite{jeki2}, the result on quartic points
\cite[Theorem 4.5]{jeki2} assume that the number fields of degree at
most 4 over $\Q$ contain $\Q(\zeta_M)$). Because
$X_1(N,N)=X(N)\times_{\Q}\Q(\zeta_N)$ and $X_1(1,N)=X_1(N)$ their
results cover results in \cite{jeki} and \cite{bkx} explained in the
survey.
\end{obs2}

\begin{section}*{Acknowledgements}
I am very happy to thank Andreas Schweizer for all his comments and
suggestions, in particular for letting me know Lemma \ref{Lem3.9}
with its proof and also the Remark \ref{Sch2} that I reproduced
here. I am pleased to thank Xavier Xarles for several discussions
about the related theory, in particular discussions concerning the
proof of Theorem \ref{thm2.12} and noticing me the example in Lemma
\ref{lem2.13} that I also reproduced here.
\end{section}

\vspace{1cm}

Francesc Bars Cortina,\\ Depart. Matem\`atiques,\\ Universitat
Aut\`onoma de Barcelona,\\ 08193 Bellaterra. Catalonia. \\%Unfortunately Spain.\\
E-mail:
francesc@mat.uab.cat \\

\end{document}